\documentclass{amsart}
\usepackage{graphicx}
\usepackage{amssymb,amscd,amsthm,amsxtra}
\usepackage{latexsym}
\usepackage{epsfig,esint}
\usepackage{verbatim}
\newtheorem{thm}{Theorem}[section]

\newtheorem{lem}[thm]{Lemma}
\newtheorem{prop}[thm]{Proposition}
\theoremstyle{definition}
\newtheorem{defn}[thm]{Definition}
\theoremstyle{remark}
\newtheorem{rem}[thm]{Remark}
\numberwithin{equation}{section}
\newcommand{\R}{\mathbb R}

\newcommand{\be}{\begin{equation}}
\newcommand{\ee}{\end{equation}}

\newcommand{\eps}{\epsilon}

\newcommand{\p}{\partial}

\begin{document}

\title[Lipschitz regularity]{Lipschitz regularity of solutions to two-phase free boundary problems}

\author{D. De Silva}
\address{Department of Mathematics, Barnard College, Columbia University, New York, NY 10027, USA}
\email{\tt  desilva@math.columbia.edu}
\author{O. Savin}
\address{Department of Mathematics, Columbia University, New York, NY 10027, USA}\email{\tt  savin@math.columbia.edu}
\thanks{O.~S.~is supported by  NSF grant DMS-1200701.}

\begin{abstract}We prove Lipschitz continuity of viscosity solutions to a class of two-phase free boundary problems governed by fully nonlinear operators. \end{abstract}

\maketitle
\section{Introduction}

Consider the two-phase free boundary problem,
\begin{equation}
\left\{
\begin{array}{ll}
\mathcal{F}(D^{2}u)=0, & \hbox{in $B_1^+(u) \cup B_1^-(u),$} \\
\  &  \\
u_{\nu }^{+}=G(u_{\nu }^{-}), &
\hbox{on $F(u):= \partial
B^+(u) \cap B_1.$} \\
&
\end{array}
\right.  \label{fb}
\end{equation}%
Here $B_r \subset \mathbb{R}^{n}$ denotes the ball of radius $r$ centered at $0$ and 
\begin{equation*}
B_1^{+}(u):=\{x\in B_1 :u(x)>0\},\quad B_1^{-}(u):=\{x\in B_1
:u(x)\leq 0\}^{\circ },
\end{equation*}%
while $u_{\nu }^{+}$ and $u_{\nu }^{-}$ denote the normal derivatives in the
inward direction to $B_1^{+}(u)$ and $B_1^{-}(u)$ respectively. $F(u)$ is the so-called \emph{free boundary} of $u$.
$\mathcal{F}$ is a fully nonlinear uniformly elliptic operator and
the function $G: \R^+  \to \R^+$ is $C^2$ and it satisfies the usual ellipticity assumption
\be\label{G_e}\mbox{ $G(t)$ is strictly increasing} \quad \text{and $G(t) \to \infty$ as $t \to \infty$.}\ee

Our main result gives the Lipschitz continuity of a viscosity solution $u$ to \eqref{fb} under the assumption that $G(t)$ behaves like $t$ for all $t$ large. Precisely, we require the following:
\begin{equation}\label{G}
G'(t) \to 1, \quad G''(t) =O\left(\frac 1 t\right),  \quad \text{as $t \to \infty$.}
\ee
This clearly includes the case $G(t)=\sqrt{1+t^2}$, which arises in several models. 

%We refer to Section 2 for all precise definitions.

\begin{thm}\label{TLip}
Let $u$ be a viscosity solution to \eqref{fb}-\eqref{G_e} and assume that \eqref{G} holds. Then
$$\|\nabla u\|_{L^\infty(B_{1/2})} \le C\left(n, \lambda,\Lambda, G  \right)(\|u\|_{L^\infty(B_1)}+1).$$
\end{thm}

The dependence on $G$ in the constant above is determined by the rate of convergence in the limit \eqref{G}. We remark that  \eqref{G} can be relaxed to $G' \in [1-\delta, 1+\delta]$ for large values of $t$ (see Section 3.)
If $\mathcal F$ is homogeneous of degree 1, then it suffices to require that $G'$ is sufficiently close to a constant as $t \to\infty$. 

The heuristic behind Theorem \ref{TLip} is that in the regime of ``big gradients" the free boundary condition becomes a continuity (no-jump) condition for the gradient. Then, gradient estimates follow from interior $C^{1,\alpha}$ estimates for fully nonlinear elliptic equations.

The study of two-phase free boundary problems for Laplace's equation was initiated by Alt, Caffarelli and Friedman in \cite{ACF} with variational techniques. The viscosity approach was later developed by Caffarelli in the pioneer works \cite{C1, C2, C3}. 
One central question is the optimal regularity of a solution $u$.
In one-phase problems, when $u$ is restricted to be non-negative, the Lipschitz regularity of the solution is an almost straightforward consequence of the free boundary condition. However, in the two-phase case, an ad-hoc monotonicity formula was introduced in \cite{ACF} to establish Lipschitz continuity and to identify so called- blow-up limits.
Variants of this formula have been obtained for example by Caffarelli, Jerison and Kenig in \cite{CJK} and by Matevosyan and Petrosyan in \cite{MP}, with applications to two-phase free boundary problems. In a recent paper \cite{DK}, Dipierro and Karakhanyan proved Lipschitz continuity of variational solutions of a two-phase free boundary problem governed by the $p$-Laplacian for a special class of isotropic free boundary conditions, without relying on the monotonicity formula.

On the other hand, in the case when the problem has a non-variational structure no known techniques are available to prove the Lipschitz continuity of the solution and analyze blow-up limits. In \cite{CDS}, we examined this question and obtained Theorem \ref{TLip} in $2D$, in the most general case when $u$ solves two different elliptic equations in the two phases and for any elliptic free boundary condition $G$. The arguments in \cite{CDS} are however purely two-dimensional. 

Here we are concerned with the question of Lipschitz continuity of viscosity solutions to general two-phase problems in any dimension. In a forthcoming paper, we will analyze the question of the classification of global Lipschitz solutions. 

This note is organized as follows. In Section 2 we present the proof of Theorem \ref{TLip}. Section 3 provides the statement of some known flatness results, which are needed in the proof of our main Theorem. It also contains some applications and extensions of our main result.

\begin{comment}Our second main result characterizes blow-up limits of Lipschitz solutions. Precisely, let $u$ be a Lipschitz continuous viscosity solution to \eqref{fb}.
Assume that $0 \in F(u)$ and let $u^*$ be a blow-up limit along a subsequence $r_k \to 0$, i.e.
$$u^*(x)=\lim_{k \to \infty} \frac{u(r_k x)}{r_k}.$$
Since $u$ is Lipschitz such blow-up limit functions exist and they are Lipschitz. We require the following additional assumption on $F$:
\be\label{F} \text{$F$ is convex (or concave) and $F$ is homogeneous of degree 1.}
\ee

\begin{thm}\label {T2}
Assume $u^*$ is as above and \eqref{F} holds.

Then either $u^*$ is a two plane-solution 
\begin{equation}\label{2p}
u^*=a (x\cdot \nu)^+ - b (x \cdot \nu)^- \quad \quad \mbox{with} \quad a, b>0, 
\end{equation}
or 
\begin{equation}\label{1p}
(u^*)^- \equiv 0,
\end{equation} 
which means that $u^*$ solves the one-phase problem for $\mathcal F$.
\end{thm} \end{comment}

\section{The proof of Theorem \ref{TLip}}

We introduce the definition of viscosity solution to our free boundary problem,
\begin{equation}
\left\{
\begin{array}{ll}
\mathcal{F}(D^{2}u)=0, & \hbox{in $B_1^+(u) \cup B_1^-(u),$} \\
\  &  \\
u_{\nu }^{+}=G(u_{\nu }^{-}), &
\hbox{on $F(u)$.} \\
&
\end{array}
\right.  \label{fb2}
\end{equation}%
$\mathcal F$ is a uniformly elliptic operator, that is 
there exist $0<\lambda\leq \Lambda$ such that for every  $M, N\in \mathcal{S}^{n\times n},$  with  $N\geq 0,$
$$
\lambda\| N\|\leq \mathcal{F}(M+N)-\mathcal{F}(M)\leq \Lambda \| N\|,
$$
where $\mathcal{S}^{n\times n}$ denotes the set of real  $n\times n$ symmetric matrices. We write $N \geq 0,$ whenever $N$ is non-negative definite and we denote by $\|M\|= \sup_{|x|=1} |Mx|$. Finally, we assume that  $\mathcal{F}(0)=0.$ The class of all such operators is denoted by $\mathcal E(\lambda, \Lambda).$

We start with some standard notion.
Given $u, \varphi \in C(B_1)$, we say that $\varphi$
touches $u$ by below (resp. above) at $x_0 \in B_1$ if $u(x_0)=
\varphi(x_0),$ and
$$u(x) \geq \varphi(x) \quad (\text{resp. $u(x) \leq
\varphi(x)$}) \quad \text{in a neighborhood $O$ of $x_0$.}$$ If
this inequality is strict in $O \setminus \{x_0\}$, we say that
$\varphi$ touches $u$ strictly by below (resp. above).

Let $\mathcal F \in \mathcal E(\lambda, \Lambda)$. If $v \in C^2(O)$, $O$ open subset in $\R^n,$ satisfies  $$\mathcal F(D^2 v) > 0  \  \ \ (\text{resp}. <0)\quad \text{in $O$,}$$ we call $v$ a (strict) classical subsolution (resp. supersolution) to the equation $\mathcal F(D^2 v) = 0 $ in $O$.

We recall that $u \in C(O)$ is a viscosity solution to $$\mathcal F(D^2 v) = 0  \quad \text{in $O$,}$$ if $u$ cannot be touched by  above (resp. below) by a strict classical subsolution (resp. supersolution) at an interior point $x_0 \in O.$

We now turn to the free boundary condition. We point out that our Theorem holds if we require the free boundary condition to be satisfied only when $u_\nu^-  \neq 0$ is large.  The precise definition is the following.

\begin{defn}\label{def} We say that  $u$ satisfies the free boundary condition 
$$ u_\nu^+=G(u_\nu^-), $$
at a point $y_0 \in F(u)$ if for any unit vector $\nu$, there exists no function $\psi \in C^2 $ defined in a neighborhood of $y_0$ with $\psi(y_0)=0$, $\nabla \psi(y_0)=\nu$ such that either of the following holds:

(1) $a \psi^+ - b \psi^- \le u$ with  $a>0$, $b > 0$  and $a> G(b)$ (i.e. $u$ is a supersolution);

(2) $a \psi^+ - b \psi^- \ge u$ with $a>0$, $b > 0$ and $a < G(b)$ (i.e. $u$ is a subsolution).
\end{defn}

We only use comparison functions which cross the $0$ level set transversally and therefore have a nontrivial negative part. For this reason the free boundary condition is preserved when taking uniform limits. It is straightforward to check that a uniform limit of solutions of \eqref{fb2} satisfies \eqref{fb2} as well.

\begin{comment}
\begin{defn}A {\it two-plane solution} $p$ to \eqref{fb2} is given by
$$p(x)=p_{x_0,\nu,a,b}(x):=a((x-x_0)\cdot \nu)^+- b((x-x_0)\cdot \nu)^-,$$
for some $x_0 \in \R^n$, $\nu \in S^1$ and with $$a=G(b), \quad \quad \mbox{and} \quad a>0, b>0.$$\end{defn}

From the viscosity definition above and the Hopf lemma we see that a two-plane solution $p$ cannot touch $u$ by above (or below) on the free boundary $\{p=0\}$ unless $u$ and $p$ coincide. As a consequence we obtain the following maximum principle.
  
\begin{lem}[Maximum principle] \label{l1}
Let $\Omega$ be a bounded domain and $p$ a two-plane solution. If $u\le p$ (or $u \ge p$) on $\p \Omega$ then $u\le p$ (respectively $u \ge p$) in $\Omega$. 
\end{lem}
\end{comment}

Our first preliminary result gives the H\"older continuity of viscosity solutions. It holds in fact for an even more general class of problems. 

From now on, constants that depend only on $n, \lambda, \Lambda,$ will be called universal and dependence on such parameters will not be specified.

\begin{thm}\label{holder} Let $u$ be a viscosity solution to  
\begin{equation}
\left\{
\begin{array}{ll}
a_{ij}u_{ij}=0, & \hbox{in $B_1^+(u) \cup B_1^-(u),$} \\
\  &  \\
u_{\nu }^{+}=G(u_{\nu }^{-}), &
\hbox{on $F(u),$}\\
&
\end{array}
\right.  \label{fbmeas}
\end{equation} with $a_{ij} \in C(B_1),$
 and assume that for $\sigma>0$
\be\label{G2} \sigma^{-1}t \geq G(t) \geq \sigma t, \quad \text{for $t>M$}. \ee
Then $u \in C^{0,\alpha}(B_{1/2})$ for some $\alpha>0$ depending on $\sigma$, and $$\|u\|_{C^{0,\alpha}(B_{1/2})} \leq C(\sigma,M)( \|u\|_{L^\infty(B_1)}+1).$$
\end{thm}
\begin{proof} Assume first that $0\in F(u)$. After dividing by a constant depending on $M$ and  $\|u\|_{L^\infty(B_1)}$ we can assume that \eqref{G2} holds for all $t > m$, with $m >0$ a sufficiently small constant, and that $\|u\|_{L^\infty(B_1)} \leq 1.$

We wish to prove the following claim.

\smallskip

{\it Claim.} There exists a constant $\delta$ depending on $\sigma,$ such that
\be \label{claim}
 \text{if $\|u\|_{L^\infty(B_1)} \leq 1$ then $\|u\|_{L^\infty(B_\delta)} \leq 1-\delta.$}
\ee
Once the claim is established, we obtain by rescaling that if $\|u\|_{L^\infty(B_1)} \leq 1$ then 
\be \|u\|_{L^\infty(B_r)} \leq r^\alpha, \quad \text{for $r=1, \delta, \delta^2, \ldots...$}
\ee
which yields the desired H\"older continuity. Notice that after the rescaling
$$\tilde u(x)= \frac{u(r x)}{r^\alpha}, \quad x \in B_1,$$
the corresponding function $$\tilde G(t) = r^{1-\alpha} G(r^{\alpha-1}t)$$ giving the free boundary condition for $\tilde u$, will still satisfy $\eqref{G2}$ for $t>m.$
 
To prove the claim, we observe first that $u^+$ and $u^-$ are subsolutions to $a_{ij}u_{ij}=0$ in $B_1$ and $0 \leq u^\pm \leq 1$. Hence, by weak Harnack inequality either of the following happens:
\be \label{weak1}\text{If \quad $\frac{|\{u^- = 0\} \cap B_r|}{|B_r|} \geq \frac 12 \quad $ then \quad $\sup_{B_{r/2}} u^- \leq (1-\eta)osc_{B_r}u^-$}
\ee
\be\label{weak2} \text{If \quad $\frac{|\{u^+ = 0\} \cap B_r|}{|B_r|} \geq \frac 12 $\quad  then \quad $\sup_{B_{r/2}} u^+ \leq (1-\eta) osc_{B_r}u^+$.}
\ee
We apply this alternative to a sequence of radii, $r_k= 2^{-k}, k=0,1,\ldots$ Say that at $k=0$, \eqref{weak1} holds. We distinguish two cases. 

Fix $\bar k$ depending on $\sigma$, to be specified later.

\smallskip

{\it Case 1.} For some $k \leq\bar k$, \eqref{weak2} holds for $r=r_k.$ Then the claim immediately follows.

\smallskip

{\it Case 2.} For all $k \leq \bar k$, \eqref{weak1} holds. Thus,
\begin{equation}\label{weak3} u^- \leq (1-\eta)^{\bar k} \leq \eps_0, \quad u^+ \leq 1 \quad \text{in $B_{r_{\bar k}}$,}\ee
with $\eps_0$ small to be made precise later (and $\bar k$ large enough depending on $\eps_0$). 

We want to show that in this case, 
\be\label{u+} u^+ \leq \frac 1 2 \quad \text{in $B_{r_{\bar k}/4}$.}\ee

Assume by contradiction that there is $x_0 \in B_{r_{\bar k}/4}$ such that $$u(x_0) > \frac 1 2.$$ Let $B_d(x_0)$ be the largest ball around $x_0$ which is contained in $B^+_{1}(u),$ i.e $$d=dist(x_0, F(u))=|x_0-y_0| \leq r_k/4.$$

Then, by Harnack inequality,
\be\label{harnack}u \geq c_0 \quad \text{in $B_{d/2}(x_0)$},\ee
with $c_0$ universal. Let 
\be\label{psi} \psi(x):= \begin{cases}c(|x-x_0|^{-\gamma} - d^{-\gamma}) & \text{if $|x-x_0| \geq d/2$}\\
c_0 & \text{if $|x-x_0| \leq d/2$,}
\end{cases}\ee
with $c$ chosen so that $\psi$ is continuous on $\p B_{d/2}(x_0)$.
Choose $\gamma$ large enough (universal) so that $a_{ij} \psi_{ij} >0$ outside $B_{d/2}(x_0).$ 

Set,
\be\label{w} w:= \psi^+ - \frac \sigma 2 \psi^-\ee
with $\sigma$ the constant in \eqref{G2}  and let $$D= B_{2d}(x_0) \setminus \overline{B_{d/2}(x_0)} \subset B_{r_{\bar k}}.$$

We claim that $u \geq w$ on $D$, if $\eps_0$ is sufficiently small. We will then contradict Definition \ref{def}-(1), if $m$ is chosen small enough. Indeed, $u$ and $w$ touch at $y_0 \in F(u)$ and in view of \eqref{G2} $$|\nabla \psi(y_0)|> G(\frac \sigma 2|\nabla \psi (y_0)|),$$
as long as
$$\frac \sigma 2 |\nabla \psi (y_0)|= \frac \sigma 2 c d^{-\gamma-1} > m.$$

We are left with the proof of our claim. The fact that $u \geq w$ in $\overline{B_{d}(x_0)}$ follows immediately by the maximum principle (see \eqref{harnack}.) Clearly, $u \geq w$ in $\{u\geq 0\} \cap (B_{2d}(x_0) \setminus \overline{B_{d}(x_0)}).$ It remains to show that $u \geq w$ in the set $\{u< 0\} \cap (B_{2d}(x_0) \setminus \overline{B_{d}(x_0)}).$ In order to apply the maximum principle we only need to show that $u \geq w$ on $\p B_{2d}(x_0) \cap \{u<0\}.$ We use that in this set, $u \geq -\eps_0$ by \eqref{weak3}. Hence it is enough to choose $\eps_0$ small enough (depending on $\sigma$), for the desired bound to hold.

\smallskip

Suppose now that $0 \not \in F(u).$ If $B_{1/2} \cap F(u) =\emptyset$ then we use interior estimates for fully nonlinear equations. If there is $x_0 \in B_{1/2} \cap F(u)$, we apply the argument above in $B_{1/2}(x_0)$ and obtain the desired 
H\"older bound in $B_{1/4}(x_0).$ We then combine this bound with interior estimates and a covering argument and obtain the desired claim.
\end{proof}

\begin{rem} If \eqref{G2} is not satisfied then the proof above shows that we can still obtain a uniform modulus of continuity $\omega$ of a viscosity solution $u$, with $\omega$ depending on $G$ and $\|u\|_{L^\infty(B_1)}$.
\end{rem}

Having established Theorem \ref{holder}, we can now prove the key Proposition in the proof of Theorem \ref{TLip}.

\begin{prop}\label{main}  Let $u$ be a viscosity solution to \eqref{fb} in $B_2$ and assume that \eqref{G} holds and $0 \in F(u)$. There exist constants $L_0, \delta,$ (depending on $G$) such that one of the following alternative holds:
\begin{enumerate}
\item $u$ is Lipschitz in $B_\delta$ and $|\nabla u| \leq C \max\{\|u\|_{L^\infty(B_1)},L_0 \}$ in $B_\delta$, with $C$ universal.
\item $\frac 1 \delta \|u\|_{L^\infty(B_{\delta})} \leq \frac 1 2 \|u\|_{L^\infty(B_1)}.$
\end{enumerate}

\end{prop}

We first need the following compactness lemma.

\begin{lem}\label{compact} Let $u_k$ be a viscosity solution to \begin{equation}
\left\{
\begin{array}{ll}
\mathcal{F}_k(D^{2}u_k)=0, & \hbox{in $B_1^+(u_k) \cup B_1^-(u_k),$} \\
\  &  \\
(u_k)_{\nu }^{+}=G_k((u_k)_{\nu }^{-}), &
\hbox{on $F_k(u_k)$}, \\
&
\end{array}
\right.  \label{fbk}
\end{equation} with $\mathcal{F}_k \in \mathcal E(\lambda, \Lambda)$ and $G_k$ satisfying \eqref{G_e}. Assume that the following convergences hold uniformly on compacts
\begin{eqnarray} \label{F_k} &\mathcal F_k \to \mathcal F^*, \\
\label{G_k}  & G_k \to  G^*, \quad \text{$G^*(t)=t$}\\
\label{u_k} & u_k \to u^*.
 \end{eqnarray}
Then 
\be\label{full} \mathcal F^*(D^2 u^*) = 0 \quad \text{in $B_1$}.\ee
\end{lem}
\begin{proof} It is standard to obtain that (see Proposition 2.9 in \cite{CC}.)
\be\label{int}\mathcal F^*(D^2 u^*) = 0 \quad \text{in $B_1 \cap \{u^*\neq 0\}$}.\ee

We next verify that the equation holds also across $\{u^*=0\}.$ Precisely, 
we need to show that if $P$ is a quadratic polynomial with $\mathcal F^*(D^2 P) >0,$ then $P$ cannot touch $u^*$ strictly by below at a point $x^* $ where $u^*(x^*)=0.$ Assume by contradiction that such a point exists.

We distinguish two cases.

\smallskip

{\it Case 1.} $\nabla P(x^*) \neq 0.$ Say for simplicity, $\nabla P(x^*)= \gamma e_n, \gamma >0.$

Set,
$$\ \psi:=(1+\eps) P^+ - P^-.$$
For $\eps$ small enough, $\psi$ still separates strictly from $u^*$ on the boundary of a small neighborhood of $x^*$, say $\p B_\rho$, and coincides with it at $x^*$.
Let $$\psi_t(x)= \psi(x+ t e_n), \quad x \in B_\rho.$$ Then for $t=-C \eps$, $C$ large, we have that $\psi_t$ is strictly below all $u_k$'s with $k$ large enough (see \eqref{u_k}).  We increase $t$ till a small $c_0 >0$ to guarantee that $\psi_t$ crosses $u^*$ and hence all the $u_k$'s with $k$ large. Thus $\psi_t$ must touch the $u_k$'s for the first time, say at $t=t_k$ small. Since the separation of $\psi$ and $u^*$ on $\p B_\rho$ is strict, the first touching point $x_k$ cannot occur there (if $c_0$ is small depending on the fix separation.)
Since $\mathcal F^*(D^2P)>0$  and $\mathcal F_k$ tends to $\mathcal F^*$ uniformly, we conclude that $x_k \in F(u_k).$ However, in view of \eqref{G_k}, $$(1+\eps)|\nabla P(x_k+ t_k e_n)| > G_k(|\nabla P(x_k+ t_k e_n)|)$$ and we contradict Definition \ref{def}-(1) for $u_k.$

\smallskip

{\it Case 2.} $\nabla P(x^*) = 0.$ Without loss of generality we can assume that $\nabla P$ vanishes only at one point, say at $x^*=0$.  Since $P$ touches $u^*$ strictly by below at $0$, $P$ and $u^*$ separate a fixed amount $s$ on the boundary of a small neighborhood $B_\rho$ of $0$. 

We translate the graph of $P$ by $t \leq \bar t $ in the negative $e_{n+1}$ direction and call $P_t=P-t$. Here $\bar t > 0$ depends on the separation $s$ of $P$ and $u^*$. 

Given a direction  $\tau$ in $\R^n$, let $$Q^t_{\tau,\eps}(x)=P_t(x+\eps \tau).$$
We choose $\eps$ small enough so that $Q^{\bar t}_{\tau, \eps}$ is strictly below $u^*$ in $B_\rho$, and all $Q^t_{\tau, \eps}$ with say $-2\bar t \leq t \leq \bar t$ separate strictly from $u^*$ on $\p B_\rho.$

We slide  $Q^t_{\tau,\eps}(x)$ till it touches $u^*$ for the first time. By the strict separation, the first touching point cannot occur on $\p B_\rho$. By \eqref{int} and the argument in Case 1, we conclude that the touching point occurs where $u^*=0$ and $\nabla Q^t_{\tau,\eps}$ vanishes. Since this holds for all $\eps$ small and all directions $\tau$ we conclude that $u^*\equiv 0$ in a neighborhood of 0, and contradict that $\mathcal F^*(D^2P) >0.$
\end{proof}

We can now easily deduce Proposition \ref{main}.

\smallskip

{\it Proof of Proposition $\ref{main}$.} Let $\delta$ be fixed, to be specified later. Assume by contradiction that there exist a sequence of constants $L_k \to \infty$, and a sequence of operators $\mathcal F_k$ and solutions $u_k$ to \eqref{fb2} such that
 $u_k$ does not satisfy neither option $(i)$ nor $(ii).$
Call, $$\max\{\|u_k\|_\infty, L_k\} =C_k$$ and let 
$$\tilde u_k = \frac{u_k}{C_k}, \quad \tilde {{\mathcal F}_k}(M):= \frac{1}{C_k} \mathcal F_k(C_k M), \quad \tilde G_k:= \frac{G(C_k t)}{C_k}.$$ Then by Theorem \ref{holder}, the uniform ellipticity of the $\mathcal F_k$'s and the first assumption in \eqref{G} on $G,$ we conclude that (up to extracting subsequences), 
\begin{eqnarray*}  & \tilde{{\mathcal F}_k} \to \mathcal F^*, \\
  & \tilde G_k \to  G^*, \quad \text{$G^*(t)=t$}\\
& \tilde u_k \to u^*,
 \end{eqnarray*} uniformly on compacts.
 
Then, by the compactness result Lemma \ref{compact} we obtain that  
\be\label{full*} \mathcal F^*(D^2 u^*) = 0 \quad \text{in $B_1$}.\ee

Hence by $C^{1,\alpha}$ estimates (see \cite{CC}) we get that 
\be\label{c1}|u^* - l| \leq C r^{1+\alpha}, \quad \forall r \leq 1,\ee
where $l(x)=a \cdot x$ for a vector $a \in \R^n$ with $|a| \leq C$ universal.

We distinguish two cases. 

\smallskip

{\it Case 1.} $|a| \leq \frac 1 4.$

In this case, clearly \eqref{c1} implies that $$\frac {1}{\delta}|u^*| \leq \frac 1 4 + C\delta^\alpha \leq \frac 1 3 \quad \text{in $B_\delta$.}$$
Thus all $u_k$'s with $k$ large satisfy $(ii)$, a contradiction.

\smallskip

{\it Case 2.} $|a| > \frac 1 4$

In this case we will use the flatness result of \cite{DFS}, which we restate in Section 3 (see Theorem \ref{flatness}).

Using that $\tilde u_k$ converges uniformly to $u^*$ and \eqref{c1} holds, we have that
\be\label{c1k}|\tilde u_k - a\cdot x| \leq 4\delta^{1+\alpha}, \quad \text{in $B_{2\delta}$}.\ee
Set, 
$$\beta_k:=|a|, \quad \alpha_k = \tilde G_k(|a|), \quad \omega:= \frac{a}{|a|}.$$
Then, \eqref{c1k} together with the fact that $\tilde G_k$ converges uniformly to the identity on compacts, yield
$$|\tilde u_k - U_{\beta_k}| \leq 5\delta^{1+\alpha}, \quad \text{in $B_{2\delta}$}$$
with
$$U_{\beta_k}:= \alpha_k(x \cdot \omega)^+ - \beta_k(x \cdot \omega)^-.$$

We conclude from Theorem \ref{flatness} in the next section that $F(\tilde u_k)$ is $C^{1,\gamma}$ in $B_{\delta}$, with universal bound independent of $k$, as long as $\delta$ is small universal. We notice that in order to apply Theorem \ref{flatness} we need to use the second relation in assumption \eqref{G} to guarantee that $\tilde G'_k$ has a universal Lipschitz modulus of continuity away from the origin.

Moreover $\tilde u_k$ is $C^1$ up to the free boundary from either side. In particular the $\tilde u_k$'s are Lipschitz with universal bound hence
$$|\nabla u_k| \leq C C_k \quad \text{in $B_\delta$}.$$
This contradicts the fact that the $u_k$'s do not satisfy $(i).$

\qed

Finally, the proof of  our main result easily follows. \smallskip

{\it Proof of Theorem $\ref{TLip}$.} In what follows $\delta, C, L_0$ are the universal constants in Proposition \ref{main}.

Say $0 \in F(u)$ and call $L:= \max\{\|u\|_{L^\infty(B_{3/4})},L_0\}.$ Set
$$a(r):=\frac 1 r \|u\|_{L^\infty(B_r)}, \quad r \leq 3/4$$
and let us show that
\be\label{claim1} a(\delta^k) \leq C L, \quad \forall k\geq 1.\ee

%$$a(r) \leq \frac{CL}{\delta} \quad \forall r\leq 1.$$ Clearly this holds at $r=1$.
By Proposition \ref{main} either alternative $(i)$ or $(ii)$ holds. In the first case, $u$ is Lipschitz in $B_\delta$ and 
$$|\nabla u| \leq C L \quad \text{in $B_{\delta}$}$$
hence our claim is clearly satisfied for all $k \geq1.$ 
%If $\delta < r \leq 1$, the claim is also trivially satisfied hence we are done.

If $(ii)$ holds, then 
$$ a(\delta) \leq \frac 1 2 \|u\|_{L^\infty(B_{3/4})} \leq L.$$
We now rescale and iterate. Call
$$u_k(x):= \frac{u(\delta^k x)}{\delta^k}, \quad k\geq 1, \quad x \in B_1.$$
Notice that $G$ remains invariant under this rescaling, hence the $u_k$'s satisfy the conclusion of Proposition \ref{main}.

If the $u_k$'s satisfy indefinitely the second alternative $(ii)$ of Proposition \ref{main}, then 
\be\label{claim2}a(\delta^k) \leq L, \quad \forall k\geq 1,\ee
as desired.

If  $\bar k \geq 1$ is the smallest $k$ for which $u_k$ does not satisfy $(ii)$, then \eqref{claim2} holds for all $1 \leq k \leq \bar k$ (hence so does \eqref{claim1}) and 
$u_{\bar k}$ satisfies the first alternative $(i)$ of Proposition \ref{main}. This means that $u_{\bar k}$ is Lipschitz in $B_{\delta},$ with
$$|\nabla u_{\bar k}| \leq C \max\{\|u_{\bar k}\|_{L^\infty(B_{3/4})}, L_0\} \quad \text{in $B_{\delta}.$}$$
Thus, using \eqref{claim2} for $k=\bar k$,
$$|\nabla u_{\bar k}| \leq C \max\{\frac{1}{\delta}\|u_{\bar k-1}\|_{L^\infty(B_\delta)}, L_0\} \leq CL \quad \text{in $B_\delta$,}$$
from which we deduce that \eqref{claim1} holds also for all $k \geq \bar k +1.$

Having established \eqref{claim1}, it is immediate that
$$|u(x)| \leq \frac{CL}{\delta} dist(x, F(u)), \quad \forall x \in B_{3/4}.$$
From this the Lipschitz continuity of $u$ in $B_{1/2}$ easily follows.

\section{Flatness results and generalizations}

\subsection{Flatness result.} In this subsection, we state a flatness Theorem that is implicitly contained in the work \cite{DFS}, although it is not explicitly formulated in this precise form. Precisely, denote by
$$U_\beta(t)= \alpha t^+ - \beta t^-, \quad \alpha=G(\beta), \quad \beta \geq 0,$$
a so-called two-plane function. The following result holds.

\begin{thm}\label{flatness} Let $u$ be a viscosity solution to \eqref{fb}-\eqref{G_e} satisfying \be\label{flat} |u(x)-U_\beta(x)| \leq \eps \quad  \text{in $B_1$}, \quad  0 < \beta_0 \leq \beta \leq \beta_1.\ee There exists a universal constant $\bar \eps$ such that if $\eps \leq \bar \eps$
then $F(u)$ is $C^{1, \gamma}$ in $B_{1/2}$, for a small universal $\gamma,$ and the $C^{1,\gamma}$ norm of $F(u)$ is bounded by a universal constant.
\end{thm}

Here a constant is called universal if it depends only on $n, \lambda, \Lambda, \beta_0, \beta_1,$ and the modulus of continuity 
of $G'$ on $[\beta_0/2, 2\beta_1].$

First notice that assumption \eqref{flat} implies that ($C$ universal)
\be\label{newflat}U_\beta(x_n - C\eps e_n) \leq u(x) \leq U_\beta(x_n+C\eps e_n), \quad \text{in $B_1$}.\ee
Thus, since $\beta \geq \beta_0>0$, one can apply Lemma 5.1 in \cite{DFS} (non-degenerate improvement of flatness) indefinitely and obtain the desired result (as long as $\bar \eps$ is small universal.)
For the reader convenience, we report Lemma 5.1 in \cite{DFS} below.

\smallskip

{\bf Lemma 5.1. \cite{DFS} }{\it Let $u$  satisfy
\begin{equation}\label{flat_1}U_\beta(x_n -\eps) \leq u(x) \leq U_\beta(x_n +
\eps) \quad \text{in $B_1,$} \quad 0\in F(u),
\end{equation} with $0 <  \beta \leq L.$

If $0<r \leq r_0$ for $r_0$ universal, and $0<\eps \leq \eps_0$ for some $\eps_0$
depending on $r$, then

\begin{equation}\label{improvedflat_2_new}U_{\beta'}(x \cdot \nu_1 -r\frac{\eps}{2})  \leq u(x) \leq
U_{\beta'}(x \cdot \nu_1 +r\frac{\eps }{2}) \quad \text{in $B_r,$}
\end{equation} with $|\nu_1|=1,$ $ |\nu_1 - e_n| \leq \tilde C\eps$ , and $|\beta -\beta'| \leq \tilde C\beta \eps$ for a
universal constant $\tilde C.$}

\smallskip

 We notice that in the Lemma above $\beta \leq L,$ where in \cite{DFS} $L$ denotes the Lipschitz constant of $u$. However the Lipschitz continuity of $u$ is not used in the proof, while it is only needed to have an upper bound for $\beta$ (and all universal constants will depend on such upper bound.) 
 
 In \cite{DFS} the Lipschitz continuity of $u$ is used to guarantee that if one only assumes that the free boundary is flat (see Theorem 1.1 \cite{DFS}), then an appropriate rescale of $u$ will satisfy the assumption \eqref{flat_1}. In the case of Theorem \ref{flatness} above, we are already guaranteed that $u$ falls in the non-degenerate setting of Lemma 5.1, hence we only need to require that $\beta \leq \beta_1.$

\subsection{Extensions and final remarks.}  First, we remark that our theorem can be proved under milder assumptions and that the dependence on $G$ of the Lipschitz constant can be removed. Precisely, we have the following.

\begin{thm}\label{TLip_mild}Let $u$ be a viscosity solution to \eqref{fb} and assume that
\be\label{newG} |G'(b) -1| \leq \delta \quad \text{on $[M, \infty)$}\ee for a small universal $\delta=\delta(n,\lambda, \Lambda)$ and a positive constant $M$. Then $u$ is Lipschitz in $B_{1/2}$ with 
$$\|\nabla u \|_{L^\infty(B_{1/2})} \leq C(n, \lambda, \Lambda, M) (\|u\|_{L^\infty(B_1)}+1).$$
\end{thm}

This result can be obtained by a close inspection of the proof of Lemma 5.1. Indeed, it can be shown that the Improvement of Flatness Lemma 5.1 holds only under the assumption that the modulus of continuity $\rho$ of $G'$ satisfies $\rho(0^+) \leq \delta.$

Theorem \ref{TLip} can be generalized to other two-phase problems. For example, it is possible to consider functions $G$ depending also on $(\nu,x)$, with $C^1$ dependence on $\nu$ and Holder dependence on $x$, as long as \eqref{G} holds uniformly in $(\nu, x)$. 

Moreover, we can consider more general operators of the form $\mathcal F(D^2u,\nabla u, u,x)$, which are uniformly elliptic in $D^2u$. The method of Section 2 easily extends to operators which enjoy $C^{1,\alpha}$ (in fact $C^1$) estimates, when appropriately rescaled. Precisely, it is necessary that a blow-down sequence
$$\mathcal F_\eps(M,q, z, x)= \eps \mathcal F(\frac 1 \eps M, \frac 1 \eps q, \frac 1 \eps z, x)$$
admits a uniformly convergent subsequence on compacts with limit $H(M, q, z, x)$ (say for $q\neq 0$) and that solutions to $H(D^2u, \nabla u, u, x)=0$ satisfy interior $C^{1,\alpha}$ estimates. There is a vast literature on the $C^{1,\alpha}$ regularity of fully nonlinear equations. We cite the results of  Caffarelli \cite{C1}, Evans \cite{E}, Krylov \cite{K}, and Trudinger \cite{T1,T2}. 

In the case when $\mathcal F$ depends only on $M$ and $p,$ then it suffices to require that $\mathcal F$ is uniformly elliptic in $M$ for all $p$'s and that
$$|\mathcal F_q(M,q)| \leq C(1+ \frac{\|M\|}{1+|q|}).$$
A concrete example is provided by the equation (see also \cite{GT}, Chapter 15)
$$a_{ij}(\nabla u)u_{ij}=0$$
where the coefficients $a_{ij}$ are uniformly elliptic, Lipschitz in $q$, and satisfy the natural growth assumption:
\be\label{quasi}|\nabla a_{ij}(q)| \leq \frac {C}{|q|}, \quad \text{for $|q|$ large}.\ee

Similarly, one can consider critical points of an energy functional of the form 
$$J(u):= \int_{B_1} (H(\nabla u) + g(x) \chi_{\{u>0\}})dx,$$
for a given $g \in C^{0,\alpha},$ where $H(q)$ satisfies a $p$-growth condition (as in the $p$-Laplace equation),
$$\lambda I \leq D^2H(q) |q|^{p-2} \leq \Lambda I, \quad p>1,$$ see for example \cite{LN1, LN2}.
Then critical points solve a two-phase free boundary problem of the form 
\be\label{div}div(\nabla H(\nabla u))=0, \quad \text{in $B_1^+(u) \cup B_1^-(u),$}\ee
$$\nabla H(u_\nu^+ \nu) \cdot u^+_\nu \nu - \nabla H(u^-_\nu \nu) \cdot u^-_\nu \nu = g(x) \quad \text{on $F(u)$.}$$
By the implicit function theorem, one can check that the free boundary condition can be expressed as
$$u^+_\nu=G(u^-_\nu, \nu,x)$$
with $G(t, \nu, x)$ approaching $t$ at infinity, uniformly in $(\nu, x).$ After dividing \eqref{div} by $|\nabla u|^{p-2}$,
if $H(q)$ satisfies the natural $p$-growth condition  $$D^3 H(q)= O(|q|^{3-p}), \quad \text{as $|q| \to \infty$}, $$ then \eqref{quasi} is satisfied and Theorem \ref{TLip} extends to critical points of $J$.

\end{document}